\documentclass[12pt,a4paper]{amsart}
\usepackage{mathrsfs}
\usepackage[active]{srcltx}
\usepackage{enumerate}
\usepackage{amsmath,amssymb,xspace,amsthm}

\usepackage{cite}
\usepackage{graphicx}
\usepackage{amscd}
\usepackage{amsfonts}
\usepackage{color}
\usepackage[mathscr]{eucal}
\usepackage{latexsym}

\theoremstyle{definition}
\newtheorem{definition}{Definition}[section]
\newtheorem{theorem}[definition]{Theorem}
\newtheorem{lemma}[definition]{Lemma}
\newtheorem{proposition}[definition]{Proposition}
\newtheorem{corollary}[definition]{Corollary}
\newtheorem{case}{Case}

\newtheorem{claim}{Claim}
\numberwithin{equation}{section}

\newcommand{\C}{\mathbb{C}}

\newcommand{\N}{\mathbb{N}}
\newcommand{\Z}{\mathbb{Z}}

\def\rank{\mathrm{rank }}
\def\spn{\mathrm{span}}
\def\Ind{\mathrm{Ind}}

\def\0{\mathbf{0}}
\def\a{\alpha}
\def\b{\beta}
\def\l{\lambda}

\def\Vir{\mathrm{Vir}}

\date{}

\title{New irreducible tensor product modules for the Virasoro algebra (II)}
\author{Xuewen Liu, Xiangqian Guo and Jing Wang}

\begin{document}

\maketitle
\begin{abstract} In this paper, we obtain a class of Virasoro modules by taking
tensor products of the irreducible Virasoro modules
$\Omega(\lambda,\alpha,h)$ and $\Omega(\mu, \b)$ with irreducible
highest weight modules $V(\theta,h)$ or with irreducible Virasoro
modules Ind$_{\theta}(N)$ defined in \cite{MZ2}. We obtain the
necessary and sufficient conditions for such tensor product modules
to be irreducible, and determine the necessary and sufficient
conditions for two of them to be isomorphic. We also compare these
modules with other known non-weight Virasoro modules.


\vskip 6pt \noindent {\em 2010  Math. Subj. Class.:} 17B10, 17B20,
17B65, 17B66, 17B68

\vskip 11pt
\end{abstract}

\section{Introduction}
Let $\C, \Z, \Z_+$ and $\N$ be the sets of all complexes, all
integers, all non-negative integers and all positive integers
respectively. The \textbf{Virasoro algebra} $\Vir$ is an infinite
dimensional Lie algebra over the complex numbers $\C$, with basis
$\{d_i,c\,\, |\,\, i \in \Z\}$ and defining relations
$$
[d_{i},d_{j}]=(j-i)d_{i+j}+\delta_{i,-j}\frac{i^{3}-i}{12}c, \quad
i,j \in \Z,
$$
$$
[c, d_{i}]=0, \quad i \in \Z.
$$
The algebra $\Vir$ is one of the most important Lie algebras both in
mathematics and in mathematical physics, see for example \cite{KR,
IK} and references therein. The Virasoro algebra theory has been
widely used in many physics areas and other mathematical branches,
for example, quantum physics \cite{GO}, conformal field theory
\cite{FMS}, vertex algebras \cite{LL},  and so on.

The theory of weight Virasoro modules with finite-dimensional weight
spaces is fairly well developed (see \cite{KR} and references
therein). In particular, a classification of weight Virasoro modules
with finite-dimensional weight spaces was given by Mathieu \cite{M},
and a classification of weight Virasoro modules with at least one
finite dimensional nonzero weight space was given in \cite{MZ1}.
There are some known irreducible non-weight Virasoro modules with
infinite-dimensional weight space was constructed by taking the
tensor product of non-weight Virasoro modules modules and some
highest modules, see \cite{TZ1}. Recently, many authors constructed
many classes of simple non-Harish-Chandra modules, including simple
weight modules with infinite-dimensional weight spaces (see
\cite{CGZ, CM, LLZ, LZ2}) and simple non-weight modules (see
\cite{BM, LGZ, LLZ, MW, MZ1, TZ1, TZ2}).

The purpose of the present paper is to construct new irreducible
non-weight Virasoro modules by taking tensor product of several
known irreducible Virasoro modules defined in \cite{CG} and
\cite{MZ2}. The present paper is organized as follows. In section 2,
we recall some important definitions and results. In section 3, we
obtain a class of irreducible modules over $\Vir$ by taking the
tensor products of $\Omega(\lambda,\alpha,h),\Omega(\mu, \b)$ and the irreducible
module Ind$_{\theta}(N)$.
In section 4, we determine the necessary and sufficient conditions
for two irreducible modules $\otimes_{i=1}^{n}\Omega(\lambda_{i},\alpha_{i},h_{i})\otimes_{j=1}^{m}\Omega(\mu_{j},\b_{j})\otimes V$ and
$\otimes_{i'=1}^{n'}\Omega(\lambda'_{i'},\alpha'_{i'},h'_{i'})\otimes_{j'=1'}^{m'}\Omega(\mu'_{j'},\b'_{j'})\otimes V'$ to be isomorphic.
In section 5, we compare the tensor
product of $\otimes_{i=1}^{n}\Omega(\lambda_{i},\alpha_{i},h_{i})\otimes_{j=1}^{m}\Omega(\mu_{j},\b_{j})\otimes V$ with all other known
non-weight irreducible modules, and obtain that $\otimes_{i=1}^{n}\Omega(\lambda_{i},\alpha_{i},h_{i})\otimes_{j=1}^{m}\Omega(\mu_{j},\b_{j})\otimes V$
is a new module.

\section{Preliminaries}
Firstly, we recall the definition of the Virasoro modules
$\Omega(\l, h, \a)$,$\Omega(\mu,\b)$, $V(\theta,h)$ and $\Ind_{\theta}(N)$ and some
basic properties of them.

\begin{definition}
Fix any $\lambda \in \C^*=\C\setminus\{0\}$, $\a \in \C$ and
$h(t)\in\C[t]$. Let $\Omega(\l, h, \a)=\C[t,\partial]$ as a vector space
and we define the $\Vir$-module action as follows:
\begin{equation*}
d_m\big(\partial^if\big)=
\l^m(\partial-m)^i\Big(\Big(\partial+mh(t)-m(m-1)\a\frac{h(t)-h(\a)}{t-\a}\Big)f(t)-m
(t-m\a)f'(t)\Big),
\end{equation*}
\begin{equation*}
c\big(\partial^if\big)=0, \;\; \forall\ \ m\in\Z, i\in\Z_+,
\end{equation*}
where $f\in\C[t]$ and $f'(t)$ is the derivative of $f$ with respect
to $t$.
\end{definition}

For convenience, we define the following operators 
\begin{equation}\label{operator}
 F(f)=\frac{h(t)-h(\a)}{t-\a}f(t)-f'(t),\quad G(f)=h(\a)f+tF(f),\
 \forall\ f\in\C[t].
\end{equation}
then the module action on $\Omega(\l,h,\a)$ can be rewritten as
\begin{equation}\label{def}
 d_m\big(\partial^if\big)=\l^m(\partial-m)^i\Big(\partial f+m G(f)-m^2\a F(f)\Big),\quad\forall\ m\in\Z, i\in\Z_+, f\in\C[t].
\end{equation}
\begin{theorem}[\cite{CG}]
$\Omega(\l,\a,h)$ is simple if and only if $\deg(h)=1$ and $\alpha\neq
0$.
\end{theorem}
\begin{definition}
Fix any $\mu \in \C^*=\C\setminus\{0\}$, $\b \in \C.$  Let $\Omega(\mu, \b)=\C[\partial]$ as a vector space
and we define the $\Vir$-module action as follows:
\begin{equation*}
d_k\big(\partial^n\big)=
\mu^k(\partial-k)^n(\partial-\b k), c\partial^n=0,\;\; \forall\ \ k\in\Z, n\in\Z_+.
\end{equation*}
\end{definition}
\begin{theorem}[\cite{LZ1}]
$\Omega(\mu,\b)$ is simple if and only if $\b\neq1.$
\end{theorem}
Let $U :=U(\Vir)$ be the universal enveloping algebra of the
Virasoro algebra $\Vir.$ For any $\theta,h\in \mathbb{C},$ let
$I(\theta,h)$ be the left ideal of $U$ generated by the set
$$ \{d_{i}| i>0\}\bigcup \{d_{0}-h\cdot1, c-\theta \cdot 1\}.$$
The Verma module with highest weight $(\theta,h)$ for $\Vir$ is
defined as the quotient module $\bar{V}(\theta,h) :=U/I(\theta,h).$ It is a
highest weight module of $\Vir$ and has a basis consisting of all
vectors of the form
$$ d^{k_{-1}}_{-1}d^{k_{-2}}_{-2}\cdots d^{k_{-n}}_{-n}v_{h}; k_{-1},k_{-2},\cdots ,k_{-n}\in \mathbb{Z}_{+},n\in \mathbb{N},$$
where $v_{h} = 1+I(\theta,h).$ Each nonzero scalar multiple of
$v_{h}$ is called a highest weight vector of the Verma module. Then
we have the irreducible highest weight module
$V(\theta,h)=\bar{V}(\theta,h)/J,$ where $J$ is the maximal proper
submodule of $\bar{V}(\theta,h).$ For the structure of
$V(\theta,h),$ refer to \cite{FF}.

Denote by $\Vir_{+}$ the Lie subalgebra of $\Vir$ spanned by all
$d_{i}$ with $i\geq 0.$ For $n\in \mathbb{Z}_{+},$ denote by
$\Vir^{(n)}_{+}$ the Lie subalgebra of $\Vir$ generated by all
$d_{i}$ for $i>n.$ For any $\Vir_{+}$ module $N$ and $\theta \in
\mathbb{C},$ consider the induced module
$\Ind(N):=U(\Vir)\otimes_{U(\Vir^{+})}N,$ and denote by
Ind$_{\theta}(N)$ the module Ind$(N)/(c-\theta)\text{Ind}(N).$ These modules are used to give a characterization of the irreducible
$\Vir$-modules such that the action of $d_k$ are locally finite for
sufficiently large $k$.

\begin{theorem}[\cite{MZ2}]\label{MZ2-1}
Assume that $N$ is an irreducible $\Vir_{+}$-module such that there
exists $k\in \N$ satisfying the following two conditions:
\begin{itemize}
\item[(a)] $d_{k}$ acts injectively on $N$;
\item[(b)] $d_{i}N=0$ for all $i>k$.
\end{itemize}
Then for any $\theta\in\C$ the $\Vir $ module $\Ind_{\theta}(N)$ is
simple.
\end{theorem}

\begin{theorem}[\cite{MZ2}]\label{MZ2-2}
Let $V$ be an irreducible $\Vir$ module. Then the following
conditions are equivalent:
\begin{itemize}
\item[(1)] There exists $k\in\N$ such that $V$ is a locally finite $\Vir^{(k)}_{+}$-module;
 \item[(2)] There exists $n\in\N$ such that $V$ is a locally nilpotent $\Vir^{(n)}_{+}$-module;
\item[(3)] Either $V$ is a highest weight module or $V\cong \Ind_{\theta}(N)$ for some $\theta\in \C$, $k\in \N$ and an irreducible $\Vir_{+}$-module $N$
satisfying the conditions $(a)$ and $(b)$ in Theorem \ref{MZ2-1}.
\end{itemize}
\end{theorem}

In the rest of the paper, we will always fix some $\Vir$-module
$\bigotimes_{i=1}^{m}\Omega(\l_{i},\a_{i},h_{i})\bigotimes_{j=1}^{n}\Omega(\mu_j,$
$\b_j)\bigotimes V$ where $\l_{i},\a_{i},\mu_j,\b_j\in\C^{*}$ and $h_{i}\in\C[t]$, and an
irreducible $\Vir$-module $V$ such that each $d_{k}$ is locally
finite (equivalently, locally nilpotent) on $V$ for any positive
integer $k$ large enough. From Theorem 2 in \cite{MZ2} we know that
$V$ has to be $V(\theta,h)$ for some $\theta,h\in \mathbb{C}$ or
Ind$_{\theta}(N)$ defined in \cite{MZ2}.

\section{Irreducible module $(\bigotimes_{i=1}^m\Omega(\lambda_i,\alpha_i,h_i)\bigotimes_{j=1}^{n}\Omega(\mu_{j},\b_{j}))\otimes V$}

In this section we will investigate the structure of the Virasoro
module $(\bigotimes_{i=1}^{m}\Omega(\l_{i},\a_{i},h_{i})$
$\bigotimes_{j=1}^{n}\Omega(\mu_{j},\b_{j}))\otimes V$
and its irreducibility.

Firstly any $f\in(\bigotimes_{i=1}^{m}\Omega(\l_{i},\a_{i},h_{i})$
$\bigotimes_{j=1}^{n}\Omega(\mu_{j},\b_{j}))\otimes V$ can
be described as a finite nonzero sum
\begin{equation}
f=\sum_{(r,p)\in S}\partial_{1}^{r_1}t_{1}^{p_1}\otimes\cdots\otimes\partial_{m}^{r_m}t_{m}^{p_m}\otimes\partial_{m+1}^{r_{m+1}}\cdots\otimes\partial_{m+n}^{r_{m+n}}\otimes v_{(r,p)}
\end{equation}
 denoted by
\begin{equation}
f=\sum_{(r,p)\in S}\partial_{1}^{r_1}t_{1}^{p_1}\cdots\partial_{m}^{r_m}t_{m}^{p_m}\partial_{m+1}^{r_{m+1}}\cdots\partial_{m+n}^{r_{m+n}}\otimes v_{(r,p)}
\end{equation}
where $(r,p)=(r_{1},r_{2},\cdots,r_{m},p_1,p_2,\cdots\cdots,p_m,r_{m+1},r_{m+2},\cdots,r_{m+n}), S\subset \Z_+^{2m+n}$ is finite and each $v_{(r,p)}\in V\backslash\{0\}$.
In order to study the module well, we should make some preparation.
 Denote
\begin{equation*}
\begin{split}
& X=\{\partial_{1}^{r_1}t_{1}^{p_1}\otimes\cdots\otimes\partial_{m}^{r_m}t_{m}^{p_{m}}\otimes\partial_{m+1}^{r_{m+1}}\otimes\cdots
\otimes\partial_{m+n}^{r_{m+n}}\otimes v_{(r,p)}:
r=(r_1,r_2,\cdots,r_{m+n})\in \Z_{+}^{m+n},\\
&\hskip1cm p=(p_1,p_2,\cdots,p_m)\in \Z_{+}^{m},v_{(r,p)}\in V\},\\
\end{split}
\end{equation*}
we can define a total order on X as follows.
$$\partial_{1}^{r_1}t_{1}^{p_1}\otimes\cdots\otimes\partial_{m}^{r_m}t_{m}^{p_{m}}\otimes\partial_{m+1}^{r_{m+1}}\otimes\cdots\otimes
\partial_{m+n}^{r_{m+n}}\otimes v_{(r,p)}<\partial_{1}^{l_1}t_{1}^{q_1}\otimes\cdots\otimes\partial_{m}^{l_m}t_{m}^{q_{m}}\otimes\partial_{l+1}^{l_{m+1}}\otimes\cdots
\otimes\partial_{l+n}^{l_{m+n}}\otimes v_{(l,q)}$$
if and only if\\
$(r_{1},r_{2},\cdots,r_{m},p_1,p_2,\cdots p_m,r_{m+1},\cdots,r_{m+n})$ and $(l_{1},l_{2},\cdots,l_{m},q_1,q_2,\cdots,q_m,l_{m+1},\cdots,l_{m+n})$
satisfy one of the following three conditions:
\begin{itemize}
\item[(1)] $\exists k\in \Z_+, r_{k}<l_{k}, \text{and}\quad\forall\ i<k, r_{i}=l_{i}$, where $1\leq k\leq m ;$
\item[(2)] $ 1\leq i\leq m,r_{i}=l_{i},\exists k\in \Z_+, p_{k}<q_{k}, \text{and}\quad\forall i<k, p_{i}=q_{i}$, where $1\leq k\leq m;$
\item[(3)] $1\leq i\leq m,r_{i}=l_{i}, p_{i}=q_{i},\exists k\in \Z_+, r_{m+k}<l_{m+k}, \text{and}\quad\forall\ j<k, r_{m+j}=l_{m+j},$
          \hskip6cm where $1\leq k\leq n ;$
\end{itemize}
Then we define the degree of $f$ $\deg(f)=(r,p)$, where
$\partial_{1}^{r_1}t_{1}^{p_1}\otimes\cdots\otimes\partial_{m}^{r_m}t_{m}^{p_m}\otimes\partial_{m+1}^{r_{m+1}}\otimes\cdots\otimes
\partial_{m+n}^{r_{m+n}}\otimes
v_{(r,p)}$ is the term with maximal order in the sum. Note that $\deg(1 \otimes\cdots\otimes1\otimes
v_{(0,0)})=\mathbf{0}=(0,0,\cdots,0).$

Next let us show two important Lemmas. The second
following technique lemma is similar to Proposition 3.2 of
\cite{CG}.
\begin{lemma} [\cite{TZ2}]\label{1}
Let $\l_1,\l_2,\cdots,\l_m\in\C, s_1,s_2,\cdots,s_m\in \N$ with $s_1+s_2+\cdots+s_m=s.$
Define a sequence of functions on $\Z$ as follows: $f_1(n)=\l_{1}^{n},f_2(n)=n\l_{1}^{n},\cdots,
f_{s_1}(n)=n^{s_{1}-1}\l_{1}^{n},f_{s_{1}+1}(n)=\l_{2}^{n},\cdots,f_{s_1+s_2}(n)=n^{s_2-1}\l_{2}^{n},\cdots,
f_s(n)=n^{s_m-1}\l_{m}^{n}.$ Let $x=(y_{pq})$ be the $s\times s$ matrix with $y_{pq}=f_{q}(p-1),
q=1,2,\cdots,s,p=r+1,r+2,\cdots,r+s$ where $r\in\Z_+$. Then
$$ det(x)=\prod_{j=1}^{m}(s_j-1)!!\l_{j}^{s_{j}(s_{j}+2r-1)/2}\prod_{1\leq i<j\leq m}(\l_{j}-\l_{i})^{s_is_j},$$
where $s_j!!=s_j!\times (s_j-1)!\times\cdots\times2!\times1!$ with
$0!!=1$, for convenience.
\end{lemma}
\begin{proposition}\label{technique}
Let $W$ be a subspace of
$(\bigotimes_{i=1}^{m}\Omega(\l_{i},\a_{i},h_{i})\bigotimes_{j=1}^{n}\Omega(\mu_{j},\b_{j}))\otimes V$ which
is stable under the action of any $d_k$ for $k$ sufficiently large.
Take any $f=\sum\limits_{(r,p)\in
S}\partial_{1}^{r_1}t_{1}^{p_1}\otimes\cdots\otimes\partial_{m}^{r_m}t_{m}^{p_m}\otimes\partial_{m+1}^{r_{m+1}}\otimes\cdots\otimes\partial_{m+n}^{r_{m+n}}\otimes
v_{(r,p)}\in W$ for some $v_{(r,p)}\in V$, then for any $0\leq s\leq
\text{max}\{r_{i}+2\}, 1\leq i \leq m$ and $0\leq l\leq
\text{max}\{r_{m+j}+1\}, 1\leq j \leq n,$
we have
\begin{equation}\label{j=general}
\begin{split}
& \hskip-0.5cm\sum_{(r,p)\in S}\sum_{i=1}^m\partial_{1}^{r_1}t_{1}^{p_1}\otimes\cdots\otimes
  \Bigg(\binom{r_i}{s}\partial_{i}^{r_{i}-s+1}t_{i}^{p_i}-\binom{r_i}{s-1}\partial_{i}^{r_{i}-s+1}G_{i}(t_{i}^{p_i})-\\
& \hskip1cm \binom{r_i}{s-2}\partial_{i}^{r_{i}-s+2}\a_{i}
  F_{i}(t_{i}^{p_i})\Bigg)\otimes\cdots\otimes t_{m}^{p_m}\partial_{m}^{r_{m}}\otimes \partial_{m+1}^{r_{m+1}}\otimes\cdots\otimes\partial_{m+n}^{r_{m+n}}\otimes v_{(r,p)}\in W,\\
 \end{split}
\end{equation}
and
\begin{equation}\label{j=general}
\begin{split}
& \sum_{(r,p)\in S}\sum_{j=1}^n\partial_{1}^{r_1}t_{1}^{p_1}\otimes\cdots\otimes
   \partial_{m}^{r_{m}}t_{m}^{p_m}\otimes \partial_{m+1}^{r_{m+1}}\otimes\\
&\hskip1cm \cdots\otimes\Bigg(\binom{r_{m+j}}{l}-\binom{r_{m+j}}{l-1}\b_j\Bigg)\partial_{m+j}^{r_{m+j}-l+1}
   \otimes\cdots\otimes\partial_{m+n}^{r_{m+n}}\otimes v_{(r,p)}\in W,\\
\end{split}
\end{equation}
where
$G\big(t_{i}^{p_i}\big)=t_{i}F\big(t_{i}^{p_i}\big)+h_{i}(\a_{i})t_{i}^{p_i}$
and we make the convention that $\binom{0}{0}=1$ and
$\binom{r_{i}}{j}=0$ whenever $j>r_{i}$ or $j<0$. In particular, we
have
\begin{enumerate}
\item \label{s=0}
when $s=0$, we have
$$\sum\limits_{(r,p)\in
S}\partial_{1}^{r_1}t_{1}^{p_1}\otimes\cdots\otimes\partial_{i}^{r_i+1}t_{i}^{p_i}\otimes\cdots\otimes\partial_{m}^{r_m}t_{m}^{p_m}\otimes
\partial_{m+1}^{r_{m+1}}\otimes\cdots\otimes\partial_{m+n}^{r_{m+n}}\otimes
v_{(r,p)}\in W, i=1,\cdots,m;$$
%
%
%
\item \label{l=0}
when $l=0$, we have $$\sum\limits_{(r,p)\in
S}\partial_{1}^{r_1}t_{1}^{p_1}\otimes\cdots\otimes\partial_{m}^{r_m}t_{m}^{p_m}
\otimes\partial_{m+1}^{r_{m+1}}\otimes\cdots\otimes\partial_{m+j}^{r_{m+j}+1}\otimes\cdots\otimes\partial_{m+n}^{r_{m+n}}\otimes
v_{(r,p)}\in W, j=1,\cdots,n;$$
\item \label{s=r_{i}+2}
when $s=r_{i}+2$, we have $$\sum\limits_{(r,p)\in
S}\partial_{1}^{r_{1}}t_{1}^{p_1}\otimes\cdots\otimes \a_{i}
F_{i}(t_{i}^{p_i})\otimes\cdots\otimes\partial_{m}^{r_{m}}t_{m}^{p_m}\otimes\partial_{m+1}^{r_{m+1}}\otimes\cdots\otimes\partial_{m+n}^{r_{m+n}}\otimes v_{(r,p)}\in W,i=1,\cdots,m;$$
\item \label{s=r_{i}+1}
when $s=r_{i}+1$, we have
\begin{equation*}
\begin{split}
&\sum\limits_{(r,p)\in S}\partial_{1}^{r_{1}}t_{1}^{p_1}\otimes\cdots\otimes G_{i}(t_{i}^{p_i})\otimes \cdots
\otimes \partial_{m}^{r_{m}}t_{m}^{p_m}\otimes\partial_{m+1}^{r_{m+1}}\otimes\cdots\otimes\partial_{m+n}^{r_{m+n}}\otimes v_{(r,p)} \in W, i=1,\cdots,m;
\end{split}
\end{equation*}
\item \label{l=r_{i}+1}
when $l=r_{m+j}+1$, we have
$$\sum\limits_{(r,p)\in S}\partial_{1}^{r_{1}}t_{1}^{p_1}\otimes \cdots\otimes \partial_{m}^{r_{m}}t_{m}^{p_m}\otimes\partial_{m+1}^{r_{m+1}}\otimes\cdots\otimes\b_{j}\otimes\cdots\otimes
\partial_{m+n}^{r_{m+n}}\otimes v_{(r,p)} \in W, j=1,\cdots,n.$$
\end{enumerate}
\end{proposition}
\begin{proof}
For any $0\neq f\in W$, we have
\begin{equation*}
\begin{split}
d_k(f)= & \sum_{(r,p)\in S}\sum_{i=1}^m\partial_{1}^{r_1}t_{1}^{p_1}\otimes\cdots\otimes
          d_k(\partial_{i}^{r_i}t_{i}^{p_i})\otimes\cdots\otimes
          \partial_{m}^{r_m}t_{m}^{p_m}\otimes\partial_{m+1}^{r_{m+1}}\otimes\cdots\otimes\partial_{m+n}^{r_{m+n}}\otimes v_{(r,p)}\\
        & +\sum_{(r,p)\in S}\sum_{j=1}^n\partial_{1}^{r_1}t_{1}^{p_1}\otimes\cdots\otimes\partial_{m}^{r_m}t_{m}^{p_m}
          \otimes\partial_{m+1}^{r_{m+1}}\otimes\cdots\otimes d_k(\partial_{m+j}^{r_{m+j}})\otimes\cdots\otimes\partial_{m+n}^{r_{m+n}}\otimes v_{(r,p)}\\
      = & \sum_{(r,p)\in S}\sum_{i=1}^m\partial_{1}^{r_1}t_{1}^{p_1}\otimes\cdots\otimes
          \l_{i}^k(\partial_i-k)^{r_i}\big(\partial_{i}t_{i}^{p_i}+kG_i(t_{i}^{p_i})-k^2\a_iF_i(t_{i}^{p_i})\big)\otimes\\
        &  \hskip5cm \cdots\otimes\partial_{m}^{r_m}t_{m}^{p_m} \otimes\partial_{m+1}^{r_{m+1}}\otimes\cdots\otimes\partial_{m+n}^{r_{m+n}}\otimes v_{(r,p)}\\
        & +\sum_{(r,p)\in S}\sum_{j=1}^n\partial_{1}^{r_1}t_{1}^{p_1}\otimes\cdots\otimes\partial_{m}^{r_m}t_{m}^{p_m}
           \otimes\partial_{m+1}^{r_{m+1}}\otimes\\
        & \hskip3cm \cdots\otimes \mu_{j}^{k}(\partial_{m+j}-k)^{r_{m+j}}(\partial_{m+j}-k\b_j)\otimes\cdots \otimes\partial_{m+n}^{r_{m+n}}\otimes v_{(r,p)}\\
      = & \sum_{(r,p)\in S}\sum_{i=1}^m \sum_{s=0}^{r_i+2}(-1)^s\l_{i}^kk^s\partial_{1}^{r_1}t_{1}^{p_1}\otimes\cdots
          \otimes \left[\binom{r_i}{s}\partial_i^{r_i-s+1}t_i^{p_i}-\binom{r_i}{s-1}\partial_i^{r_i-s+1} G_i(t_{i}^{p_i})\right.\\
        & \hskip1cm \left.-\binom{r_i}{s-2}\partial_i^{r_i-s+2}\a_iF_i(t_{i}^{p_i})\right]\otimes\cdots\otimes\partial_{m}^{r_m}t_{m}^{p_m}\otimes\partial_{m+1}^{r_{m+1}}\otimes\cdots         \otimes\partial_{m+n}^{r_{m+n}}\otimes v_{(r,p)}\\
        & +\sum_{(r,p)\in S}\sum_{j=1}^n \sum_{l=0}^{r_{m+j}+1}(-1)^l\mu_{j}^kk^l\partial_{1}^{r_1}t_{1}^{p_1}\otimes\cdots\partial_{m}^{r_m}t_{m}^{p_m}\otimes
          \partial_{m+1}^{r_{m+1}}\otimes\\
        & \hskip4cm \cdots\otimes \left[\binom{r_{m+j}}{l}+\binom{r_{m+j}}{l-1}\b_j\right] \partial_{m+j}^{r_{m+j}-l+1}\otimes\cdots\otimes\partial_{m+n}^{r_{m+n}}\otimes v_{(r,p)}
\end{split}
\end{equation*}
where $k$ is sufficiently large. From lemma \ref{1}, we see that the
coefficients of $k^{s}\l_{i}^{k}$ and $k^{l}\mu_{j}^{k}$ in
$(\bigotimes_{i=1}^{m}\Omega(\l_{i},\a_{i},h_{i})\bigotimes_{j=1}^{n}\Omega(\mu_{j},\b_{j}))\otimes V$ belong
to $W$ for $1\leq i\leq m, 0\leq s \leq \text{max}\{r_{i}+2\}$ and $1\leq j\leq n, 0\leq l\leq \text{max}\{r_{m+j}+1\}$.
Taking particular value to $ i,s,j,l$, then we get
$(1)-(5)$.
\end{proof}

\begin{theorem}Let $m, n\in\N$, $\lambda_{i}, \alpha_{i},\mu_j, \b_j\in\C^{*}$, $\deg h_{i}=1$ for $i=1,2,\cdots,m , j=1,2,\cdots,n$ and $\lambda_{1},\cdots,\l_m,\mu_1,\cdots,\mu_n$
are pairwise distinct. Let $V$ be an irreducible module over $\Vir$ such that each
$d_{k}$ is locally finite on both $V$ for all $k\geq K$
where $K$ is a fixed positive integer. Then the tensor product
$((\bigotimes_{i=1}^{m}\Omega(\l_{i},\a_{i},h_{i})\bigotimes_{j=1}^{n}\Omega(\mu_{j},\b_{j}))\otimes V$
is an irreducible $\Vir$ module.
\end{theorem}
\begin{proof}Let $X=(\bigotimes_{i=1}^{m}\Omega(\l_{i},\a_{i},h_{i})\bigotimes_{j=1}^{n}\Omega(\mu_{j},\b_{j}))\otimes V$, and let $W$ be a nonzero submodule of $X$.
Take a nonzero element $f=\sum\limits_{(r,p)\in
S}\partial_{1}^{r_1}t_{1}^{p_1}\otimes\cdots\otimes\partial_{m}^{r_m}t_{m}^{p_m}\otimes\partial_{m+1}^{r_{m+1}}\otimes\cdots\otimes
\partial_{m+n}^{r_{m+n}}\otimes
v_{(r,p)}\in W$ with minimal degree. By Proposition \ref{technique} \eqref{s=r_{i}+2}, we
have $d_k(\sum\limits_{(r,p)\in S}
\partial_{1}^{r_{1}}t_{1}^{p_1}\otimes\cdots\otimes \a_{i}
F_{i}(t_{i}^{p_i})\otimes\cdots\otimes
\partial_{m}^{r_{m}}t_{m}^{p_m}\otimes\partial_{m+1}^{r_{m+1}}\otimes\cdots\otimes
\partial_{m+n}^{r_{m+n}}\otimes v_{(r,p)})\in W$. And by Proposition \ref{technique} \eqref{l=r_{i}+1}, we get
 $\sum\limits_{(r,p)\in S}
\partial_{1}^{r_{1}}t_{1}^{p_1}\otimes\cdots\otimes \a_{i}
F_{i}(t_{i}^{p_i})\otimes\cdots\otimes
\partial_{m}^{r_{m}}t_{m}^{p_m}\otimes\partial_{m+1}^{r_{m+1}}\otimes\cdots\otimes
\b_j\otimes\cdots\otimes v_{(r,p)}\in W$.
Then from the minimality of $(r,p)$,
we have $\sum\limits_{(0,p)\in S}t_{1}^{p_1}\otimes\cdots\otimes t_{m}^{p_m}\otimes1\otimes\cdots\otimes1\otimes v_{(0,p)}\in W$.
Using the action of $d_{k}$ on$\sum\limits_{(0,p)\in S}t_{1}^{p_1}\otimes\cdots\otimes t_{m}^{p_m}\otimes1\otimes\cdots\otimes1\otimes v_{(0,p)}$,
we get
$\sum\limits_{(0,p)\in S}t_{1}^{p_1}\otimes\cdots\otimes \a_iF_{i}(t_{i}^{p_i})\otimes\cdots\otimes t_{m}^{p_m}\otimes1\otimes\cdots\otimes1\otimes v_{(0,p)}\in W,$
by Proposition \ref{technique} \eqref{s=r_{i}+2},  we
 have $\sum\limits_{(0,p)\in S}t_{1}^{p_1}\otimes\cdots\otimes t_{i}^{p_{i}-1}\otimes\cdots\otimes t_{m}^{p_m}\otimes1\otimes\cdots\otimes1\otimes v_{(0,p)}\in W.$
So $1\otimes\cdots\otimes1\otimes v_{(0,p)} \in W$
by downward induction on the degree of $t_i$. Similarly, we use $d_{k}(1\otimes \cdots\otimes1\otimes v_{(0,p)})$
and we can have
\begin{equation}\begin{array}{l}
1\otimes\cdots\otimes\partial_{i}\otimes\cdots\otimes 1\otimes v_{(0,p)}\in W;\quad
1\otimes\cdots\otimes\partial_{m+j}\otimes\cdots\otimes 1\otimes v_{(0,p)})\in W;\\\\
1\otimes\cdots\otimes G_{i}(1)\otimes\cdots\otimes 1\otimes v_{(0,p)}\in W;\quad
1\otimes\cdots\otimes \a_iF_{i}(1)\otimes\cdots\otimes 1\otimes v_{(0,p)}\in W;\\\\
\end{array}\end{equation}
where $i=1,\cdots,m, j=1,\cdots,n.$
Next again by Proposition \ref{technique} , we can deduce
$(\bigotimes_{i=1}^{m}\Omega(\l_{i},$
$\a_{i},h_{i})\bigotimes_{j=1}^{n}\Omega(\mu_{j},\b_{j}))\otimes v_{(0,p)}\in W$
by upward induction on the degree of $\partial_{i},\partial_{m+j}$ and $t_i$.
Now let
$$ Y=\{v\in V | (\bigotimes_{i=1}^{m}\Omega(\l_{i},\a_{i},h_{i})\bigotimes_{j=1}^{n}\Omega(\mu_{j},\b_{j}))\otimes v\in W\},$$
we know that $Y$ is nonempty.
Again using
$d_{k}(w\otimes v)=d_{k}w\otimes v + w\otimes d_{k}v$, where $w\in(\bigotimes_{i=1}^{m}\Omega(\l_{i},\a_{i},h_{i})\bigotimes_{j=1}^{n}\Omega(\mu_{j},\b_{j})), v\in Y,$
we deduce that $Y$ is a submodule of $V$. Thus $Y=V$ and $W=X$.
\end{proof}
From the above theorem, we see that if $\l_1,\cdots,\l_m, \mu_1,\cdots,\mu_n$ are pairwise distinct,
$\bigotimes_{i=1}^{m}\Omega(\l_{i},\a_{i},h_{i})\bigotimes_{j=1}^{n}\Omega(\mu_{j},\b_{j})$ is an irreducible $\Vir$ module.
When $\l_1,\cdots,\l_m, \mu_1,\cdots,\mu_n$ are not pairwise distinct, is  the module
$\bigotimes_{i=1}^{m}\Omega(\l_{i},\a_{i},h_{i})\bigotimes_{j=1}^{n}\Omega(\mu_{j},\b_{j})$ reducible? This
conjecture is correct.
\par
Next we will discuss the reducibility of $\Omega(\l,\a,h)\bigotimes\Omega(\l,\b)$. As before, we denote
\\
$\Omega(\l,\a,h)\bigotimes\Omega(\l,\b)=\C[\partial_1,\partial_2,t].$ Since $\C[\partial_1,\partial_2,t]$
is the polynomial algebra, we have $\Omega(\l,\a,h)=\C[\partial_1,t], \Omega(\l,\b)=\C[\partial_2]$
or $\Omega(\l,\a,h)=\C[\partial_2,t], \Omega(\l,\b)=\C[\partial_1].$ So we can have the following theorem.
\begin{theorem} Let $m\in \Z_+,$ $\l,\a,\b\in \C^{*},$ $h\in \C[t]$ and $\deg(h)=1.$
Let $W_m$ be the subspace of $\Omega(\l,\a,h)\otimes\Omega(\l,\b)$ spanned by the elements:
$\{\partial_{1}^{l}(\partial_1+\partial_2)^{n}\C[t]\mid 0\leq l\leq m, \forall n\in\Z_+\},$
where $\partial_1\in \Omega(\l,\a,h), \partial_2\in \Omega(\l,\b)$ or $\partial_1\in \Omega(\l,\a,h),\partial_2\in \Omega(\l,\b).$ Then each $W_m$ is
a submodule of $\Omega(\l,\a,h)\otimes\Omega(\l,\b).$
\end{theorem}
\begin{proof} We denote $\Omega(\l,\a,h)\otimes\Omega(\l,\b)=\C[\partial_1,\partial_2,t],$ and $\partial_1\in \Omega(\l,\a,h)$ or $\partial_1\in \Omega(\l,\a,h),$
so we reduce $W_m$ has two case.
\begin{case} $\Omega(\l,\a,h)=\C[\partial_1,t], \Omega(\l,\b)=\C[\partial_2].$ \end{case}
Take $0\neq f\in\C[t]$, for any $l,n\in\Z_+, k\in \Z,$ we have
\begin{equation*}
\begin{split}
  &   \l^{-k}d_k(\partial_{1}^{l}(\partial_1+\partial_2)^{n}f)
=    \l^{-k}d_k(\partial_{1}^{l}\sum_{j=0}^{n}\binom{n}{j}\partial_{1}^{j}\partial_{2}^{n-j}f)
=    \l^{-k}\sum_{j=0}^{n}\binom{n}{j}d_k(\partial_{1}^{l+j}f\partial_{2}^{n-j})\\
= &   \l^{-k}\sum_{j=0}^{n}\binom{n}{j}d_k(\partial_{1}^{l+j}f)\partial_{2}^{n-j}+\l^{-k}\sum_{j=0}^{n}\binom{n}{j}\partial_{1}^{l+j}fd_k(\partial_{2}^{n-j})\\
= &   \sum_{j=0}^{n}\binom{n}{j}(\partial_{1}-k)^{l+j}\Big(\partial_{1}f+kG(f)-k^{2}\a F(f)\Big)\partial_{2}^{n-j}
      +\sum_{j=0}^{n}\binom{n}{j}\partial_{1}^{l+j}f(\partial_2-k)^{n-j}\Big(\partial_2-\b k\Big)\\
= &   (\partial_{1}-k)^{l}\Big(\partial_{1}f+kG(f)-k^{2}\a F(f)\Big)\sum_{j=0}^{n}\binom{n}{j}(\partial_{1}-k)^{j}\partial_{2}^{n-j}\\
  &   \hskip4cm  +\partial_{1}^{l}f\Big(\partial_2-\b k\Big)\sum_{j=0}^{n}\binom{n}{j}\partial_{1}^{j}(\partial_2-k)^{n-j}\\
= &   (\partial_{1}-k)^{l}\Big(\partial_{1}f+kG(f)-k^{2}\a F(f)\Big)(\partial_2+\partial_1-k)^{n}+\partial_{1}^{l}f\Big(\partial_2-\b
      k\Big)(\partial_2+\partial_1-k)^{n}\\
= &   \Big[\partial_{1}\Big((\partial_{1}-k)^{l}-\partial_{1}^{l}\Big)f+\Big(kG(f)-k^{2}\a F(f)\Big)(\partial_{1}-k)^{l}\Big](\partial_2+\partial_1-k)^{n}
      +\partial_{1}^{l}f\Big(\partial_1+\partial_2-\b k\Big)\\
  &   \hskip4cm  (\partial_2+\partial_1-k)^{n},\\
\end{split}
\end{equation*}
and this implies $d_k(\partial_{1}^{l}(\partial_1+\partial_2)^{n}f)\in W_m$ for $0\leq l\leq m, n\in \Z_+, k\in \Z.$
Moreover, $c(W_m)={0}\subset W_{m}.$ So we deduce that $W_m$ is a submodule of $\Omega(\l,\a,h)\otimes\Omega(\l,\b).$
\begin{case} $\Omega(\l,\a,h)=\C[\partial_2,t], \Omega(\l,\b)=\C[\partial_1].$ \end{case}
Similarly, take $0\neq f\in\C[t]$, for any $l,n\in\Z_+, k\in \Z,$ we have
\begin{equation*}
\begin{split}
  &   \l^{-k}d_k(\partial_{1}^{l}(\partial_1+\partial_2)^{n}f)
=    \l^{-k}d_k(\partial_{1}^{l}\sum_{j=0}^{n}\binom{n}{j}\partial_{1}^{j}\partial_{2}^{n-j}f)
=    \l^{-k}\sum_{j=0}^{n}\binom{n}{j}d_k(\partial_{1}^{l+j}\partial_{2}^{n-j}f)\\
= &   \l^{-k}\sum_{j=0}^{n}\binom{n}{j}d_k(\partial_{1}^{l+j})\partial_{2}^{n-j}f+\l^{-k}\sum_{j=0}^{n}\binom{n}{j}\partial_{1}^{l+j}d_k(\partial_{2}^{n-j}f)\\
= &   \sum_{j=0}^{n}\binom{n}{j}(\partial_{1}-k)^{l+j}\Big(\partial_{1}-\b k\Big)\partial_{2}^{n-j}f
      +\sum_{j=0}^{n}\binom{n}{j}\partial_{1}^{l+j}(\partial_2-k)^{n-j}\Big(\partial_2f+kG(f)-k^{2}\a F(f)\Big)\\
= &   (\partial_{1}-k)^{l}\Big(\partial_1-\b k\Big)\sum_{j=0}^{n}\binom{n}{j}(\partial_{1}-k)^{j}\partial_{2}^{n-j}f
      +\partial_{1}^{l}\Big(\partial_2f+kG(f)-k^{2}\a F(f)\Big)\\
  &   \hskip3.5cm  \sum_{j=0}^{n}\binom{n}{j}\partial_{1}^{j}(\partial_2-k)^{n-j}\\
= &   (\partial_{1}- k)^{l}\Big(\partial_{1}-\b k\Big)(\partial_2+\partial_1-k)^{n}f+\partial_{1}^{l}\Big(\partial_2f+kG(f)-k^{2}\a
      F(f)\Big)(\partial_2+\partial_1-k)^{n}\\
= &    \Big[\partial_{1}\Big((\partial_{1}-k)^{l}-\partial_{1}^{l}\Big)-\b k(\partial_{1}-k)^{l}\Big](\partial_2+\partial_1-k)^{n}f
      +\partial_{1}^{l}\Big((\partial_1+\partial_2)f+kG(f)\\
  &   \hskip3.5cm  -k^{2}\a F(f)\Big)(\partial_2+\partial_1-k)^{n},\\
\end{split}
\end{equation*}
which also implies $d_k(\partial_{1}^{l}(\partial_1+\partial_2)^{n}f)\in W_m$ for $0\leq l\leq m, \forall n\in \Z_+, k\in \Z.$
Clearly, $c(W_m)={0}\subset W_{m}.$ Therefore, we complete the proof of this theorem.
\end{proof}

\begin{corollary}Let $\l,\a,\b \in\C^{*}$ and $m\in \Z_+.$ Denote by $W_m$ the submodule of $\Omega(\l,\a,h)\otimes\Omega(\l,\b)$
spanned by the elements $\{\partial_{1}^{l}(\partial_1+\partial_2)^{n}\C[t],0\leq l\leq m, n\in \Z_+\},$
where $\partial_1\in \Omega(\l,\a,h), \partial_2\in \Omega(\l,\b)$ or $\partial_2\in \Omega(\l,\a,h),\partial_1\in \Omega(\l,\b).$
Set $W_{-1}=0$. Then $\overline{W}_{m}=W_m/W_{m-1}$ is isomorphic to the module $\Omega(\l,\a,h-m-\b).$
\end{corollary}
\begin{proof}From the proof of the above theorem, we know the module $W_{m}$ has two cases. So we need discuss two cases.
\begin{claim} If $\Omega(\l,\a,h)=\C[\partial_1,t], \Omega(\l,\b)=\C[\partial_2], \Omega(\l,\a,h-m-\b)=\C[\partial, t]$,\\
then $\overline{W}_{m}=W_m/W_{m-1}\cong \Omega(\l,\a,h-m-\b).$ \end{claim}
When $\Omega(\l,\a,h)=\C[\partial_1,t], \Omega(\l,\b)=\C[\partial_2],$ we have
\begin{equation*}
\begin{split}
  & \l^{-k}d_k(\partial_{1}^{m}(\partial_1+\partial_2)^{n}f+W_{m-1})\\
= & \l^{-k}d_k(\partial_{1}^{m}(\partial_1+\partial_2)^{n}f)+W_{m-1}\\
= & (\partial_{1}-k)^{m}\Big(\partial_{1}f+kG(f)-k^{2}\a F(f)\Big)(\partial_2+\partial_1-k)^{n}+\partial_{1}^{m}f\Big(\partial_2-\b k\Big)(\partial_2+\partial_1-k)^{n}+W_{m-1}\\
= & \partial_{1}^{m}\Big[(\partial_{1}+\partial_{2})f+k\big(G(f)-mf-\b f\big)-k^{2}\a F(f)\Big](\partial_2+\partial_1-k)^{n}+W_{m-1}, \quad\forall\ f\in\C[t],\\
\end{split}
\end{equation*}
and $c(\overline{W}_{m})=0.$ Thus there is a $\Vir$ module isomorphic
$\varphi: \overline{W}_{m} \longrightarrow \Omega(\l,\a,h-m-\b)$
by $\partial_{1}^{m}(\partial_1+\partial_2)^{n}f+W_{m-1} \longrightarrow \partial^{n}f.$ This completes the claim 1.
\begin{claim}$\Omega(\l,\a,h)=\C[\partial_2,t], \Omega(\l,\b)=\C[\partial_1],$ $\overline{W}_{m}=W_m/W_{m-1}\cong \Omega(\l,\a,h-m-\b).$ \end{claim}
When $\Omega(\l,\a,h)=\C[\partial_2,t], \Omega(\l,\b)=\C[\partial_1],$ we have
\begin{equation*}
\begin{split}
  & d_k(\partial_{1}^{m}(\partial_1+\partial_2)^{n}f+W_{m-1})\\
= & d_k(\partial_{1}^{m}(\partial_1+\partial_2)^{n}f)+W_{m-1}\\
= & (\partial_{1}- k)^{m}\Big(\partial_{1}-\b k\Big)(\partial_2+\partial_1-k)^{n}f+\partial_{1}^{m}\Big(\partial_2f+kG(f)-k^{2}\a F(f)\Big)(\partial_2+\partial_1-k)^{n}\\
=&\partial_{1}^{m}\Big[(\partial_{1}+\partial_{2})f+k\big(G(f)-mf-\b f\big)-k^{2}\a F(f)\Big](\partial_2f+\partial_1-k)^{n}+W_{m-1}, \quad\forall\ f\in\C[t],\\
\end{split}
\end{equation*}
the following proof is similar to the first.
\end{proof}
\section{Isomorphism}
In this section we will determine the necessary and sufficient
conditions for two such modules of the form $\bigotimes_{i=1}^{m}\Omega(\lambda_i, \alpha_i, h_i )\bigotimes_{j=1}^{n}\Omega(\mu_j, \b_j )\bigotimes V$ to be isomorphic.
\par
As the paper of [TZ], we will firstly discuss some properties of $\bigotimes_{i=1}^{m}\Omega(\lambda_i, \alpha_i, h_i)\bigotimes_{j=1}^{n}\Omega(\mu_j, $ $\b_j )\bigotimes V$ denoted by $Y$, where
$\l_i, \mu_j, \a_i\in\mathbb{C^*}, h_i\in\mathbb{C}[t]$, $\l_1,\cdots, \l_m, \mu_1, \cdots, \mu_n$ are pairwise distinct, and $V$ is an irreducible module over Vir such that
$d_k$ is locally finite on $V$ for all $k>K$, where $K$ is a fixed positive integer.
 Set
$$R_f=\lim _{K\rightarrow \infty}\text{rank}(\{d_k(f): k>K\}), \quad R_Y=\text{inf}\{R_f: f\in Y\backslash\{0\}\}.$$
Then we will get
\begin{lemma} Let $f\in Y,$ we get\\
$(1)$ $R_f=\text{rank}(\{d_k(f): k>K \}), \forall K>K(f).$\\
$(2)$ If $m>1$, then $R_f=2m+n+1$ if and only if $f\in\{ 1\otimes 1\cdots\otimes 1\otimes v: v\in V\}$.\\
$(3)$ $R_Y=2m+n+1$.
\begin{proof} (1)  We take $f=\sum_{(r,p)\in S}\partial_{1}^{r_1}t_{1}^{p_1}\otimes\cdots\otimes\partial_{m}^{r_m}t_{m}^{p_m}\otimes\partial_{m+1}^{r_{m+1}}\otimes\cdots\otimes\partial_{m+n}^{r_{m+n}}\otimes v_{(r,p)}$ from $Y$  , where $\partial_{1}^{r_1}t_{1}^{p_1}\otimes\cdots\otimes\partial_{m}^{r_m}t_{m}^{p_m}\otimes\partial_{m+1}^{r_{m+1}}\otimes\cdots\otimes\partial_{m+n}^{r_{m+n}}\otimes v_{(r,p)}$ be the unique maximal term in the sum.
For each $k>K>K(f)$ we have
\begin{equation*}\begin{split}
&  d_k(f)\\
&  =\sum_{(r,p)\in S}\sum_{i=1}^m \sum_{s=0}^{r_i+2}(-1)^s\l_{i}^k k^s\partial_{1}^{r_1}t_{1}^{p_1}\otimes\cdots\otimes
   [\binom{r_i}{s}\partial_i^{r_i-s+1}t_i^{p_i}-\binom{r_i}{s-1}\partial_i^{r_i-s+1} G_i(t_{i}^{p_i})\\
&  \hskip1cm-\binom{r_i}{s-2}\partial_i^{r_i-s+2}\a_iF_i(t_{i}^{p_i}))]
   \otimes\cdots\otimes\partial_{m}^{r_m}t_{m}^{p_m}\otimes\partial_{m+1}^{r_{m+1}}\otimes\cdots\otimes
   \partial_{m+n}^{r_{m+n}}\otimes v_{(r,p)}\\
&  \hskip0.5cm+\sum_{(r,p)\in S}\sum_{j=1}^n\sum_{l=0}^{r_{m+j}+1}\partial_{1}^{r_1}t_{1}^{p_1}\otimes\cdots\otimes
   \partial_{i}^{r_i}t_{i}^{p^i}\otimes\cdots\otimes\partial_{m}^{r_m}t_{m}^{p_m}\otimes\partial_{m+1}^{r_{m+1}}\otimes\cdots\\
&  \hskip1cm \otimes k^l\mu_{j}^{k}[(-1)^l(\binom {r_{m+j}}{l}+\binom {r_{m+j}}{l-1}\b_j)]\partial_{m+j}^{r_{m+j}+1-l}\otimes
   \cdots\otimes\partial_{m+n}^{r_{m+n}}\otimes v_{(r,p)}\\
\end{split}
\end{equation*}
Denote by $g_{ks}$ the coefficient of $k^s\l_i^k$ in $Y$ for $ 1\leq i\leq m$, $0\leq s\leq2+\text{max}\{r_i\}$, $b_{kl}$ the coefficient of $k^l\mu_j^k$ for $1\leq j\leq n, 0\leq l\leq\text{ max}\{r_{m+j}\}+1$ and $T_f=\{g_{ks},b_{kl}\}$, $R_K=\{d_k(f): k>K\}$.

Then by Lemma 3.1 we have $\text{span}(R_K)=\text{span}(T_f)$, so $R_f=\text{rank}(T_f)=\text{rank}(R_K)$ is independent of $K>K(f)$.
\par
(2) Since $d_k(1\otimes v)=\sum_{i=1}^m\l_i^k(\partial_i+k(t_i\xi_i+h_i(\a_i))-k^2\a_i\xi_i)\otimes v+\sum_{j=1}^{n}{\mu_j}^k(\partial_{m+j}-k\b_j)\otimes v, \forall k>K$, considering the coefficients of $\l_{i}^{k}, k\l_{i}^{k}, k^2\l_{i}, \mu_{j}^{k}, k\mu_{j}^{k}$ by Lemma 3.1, we have $\text{rank}(d_k(1\otimes v))=2m+n+1$.
\par
 Next we will discuss $R_f$ when $f\notin\{ 1\otimes 1\otimes\cdots\otimes 1\otimes v: v\in V\}$. If $f=\partial_{1}^{r_1}t_{1}^{p_1}\cdots\partial_{m}^{r_m}t_{m}^{p_m}\partial_{m+1}^{r_{m+1}}$
 $\cdots\partial_{m+n}^{r_{m+n}}\otimes v_{(r,p)}$, by the equation $d_k(f)$ above we have some results as in the following
\begin{itemize}
 \item[Case 1] $s=0, l=0$, the coefficients of $\l_i^k$ and $\mu_j^k$ in $d_k(f)$ are
               $$\partial_{1}^{r_1}t_{1}^{p_1}\cdots\partial_i^{r_i+1}t_i^{p_i}\cdots\partial_{m}^{r_m}t_{m}^{p_m}\partial_{m+1}^{r_{m+1}}
               \cdots\partial_{m+n}^{r_{m+n}}\otimes v_{(r,p)}$$
               and
               $$\partial_{1}^{r_1}t_{1}^{p_1}\cdots\partial_{m}^{r_m}t_{m}^{p_m}\partial_{m+1}^{r_{m+1}}\cdots\partial_{m+j}^{r_{m+j}+1}
               \partial_{m+n}^{r_{m+n}}\otimes v_{(r,p)}$$
               respectively, $1\leq i\leq m, 1\leq j\leq n$.
 \item[Case 2] $s=0, l=1$, the coefficient of $k\mu_j^k$ in $d_k(f)$ is
               $$\partial_{1}^{r_1}t_{1}^{p_1}\cdots\partial_{m}^{r_m}t_{m}^{p_m}\partial_{m+1}^{r_{m+1}}\cdots[-r_{m+j}-\b_j]
               \partial_{m+j}^{r_{m+j}}\cdots\partial_{m+n}^{r_{m+n}}\otimes v_{(r,p)}, 1\leq j\leq n,$$
 \item[Case 3] $s=1, l=1$, the coefficient of $k\l_i^k$ in $d_k(f)$ is
               $$\hskip1.5cm\partial_{1}^{r_1}t_{1}^{p_1}\cdots[(h_i(\a_i)-p_i-r_i)\partial_i^{r_i}t_i^{p_i}+\xi_i\partial_i^{r_i}t_i^{p_i+1}]\cdots
               \partial_{m}^{r_m}t_{m}^{p_m}
               \partial_{m+1}^{r_{m+1}}\cdots\partial_{m+n}^{r_{m+n}}\otimes v_{(r,p)}, 1\leq i\leq m,$$
 \item[Case 4] $s=2, l=0$, the coefficient of $k^2\l_i^k$ in $d_k(f)$ is
               $$\hskip0.8cm\partial_{1}^{r_1}t_{1}^{p_1}\cdots[\binom{r_i}{2}\partial_i^{r_i-1}t_i^{p_i}-\binom{r_i}{1}\partial_i^{r_i-1}G_i(t_i^{p_i})-\
               a_i\partial_i^{r_i}F_i(t_i^{p_i})]
               \cdots\partial_{m}^{r_m}t_{m}^{p_m}\partial_{m+1}^{r_{m+1}}\cdots\partial_{m+n}^{r_{m+n}}\otimes v_{(r,p)}, 1\leq i\leq m,$$
  \item[Case 5] $s=2, l=2$, the coefficient of $k^2\mu_j^k$ is
                $$\partial_1^{r_1}t_1^{p_1}\cdots\partial_{m+j}^{r_{m+j}-1}\cdots\partial_{m+n}^{r_{m+n}}\otimes v_{(r,p)}, 1\leq j\leq n,$$
                where $G_i(t_i^{p_i})=(h_i(\a_i)-p_i)t_i^{p_i}+\xi_it_i^{p_i+1}$, $F_i(t_i^{p_i})=\xi_it_i^{p_i}-p_it_i^{p_i-1}$, $1\leq i\leq m$.
 \end{itemize}
  Since $f$ is not $1\otimes v$, $v\in V$, that is  $r_i$ and $p_i$ are not both $0$.
  By Lemma 3.1 and the above four cases, we get rank$\{d_k(f): k> K(f)\}> 2m+n+1$.
  This completes the part (2), and the part (3) is easy to be obtained.

\end{proof}

\end{lemma}
In the paper [GWL], we have known $\Omega(\l_1, \a_1, h_1)\bigotimes V_1\cong\Omega(\l_2, \a_2, h_2)\bigotimes V_2$ as Vir modules if and only if
$\l_1=\l_2, \a_1\xi_1=\a_2\xi_2$ and $V_1\cong V_2$, where $\l_i\in\mathbb{C^*}, deg(h_i)=1$, and $\a_i\neq 0$, where $i=1,2$.
Obviously, $\Omega(\l_i, \a_i, h_i)\bigotimes V_i\cong\Omega(\l_i,\a_i\xi_i, t_i)\bigotimes V_i$, $i=1,2$. Therefore,
the isomophism problem of two modules as $\bigotimes_{i=1}^{m}\Omega(\lambda_i, \alpha_i, h_i )\bigotimes_{j=1}^{n}\Omega(\mu_j, \b_j )\bigotimes V$ can be converted into the isomophism of two modules as $\bigotimes_{i=1}^{m}\Omega(\lambda_i, a_i, t_i )\bigotimes_{j=1}^{n}\Omega(\mu_j, \b_j )\bigotimes V$

\begin{proposition} \label{4.2}
Let $m, n\in\Z_+$, $\l_i, a_i, \mu_j, \b_j, \l'_i, a_i', \mu_j', \b_j'\in\mathbb{C^*}$ , $\l_1, \cdots, \l_m, \mu_1,\cdots, \mu_n$ are pairwise distinct, $\l'_1,\cdots, \l'_m, \mu'_1,\cdots, \mu'_n$ are pairwise distinct, $V_1, V_2$ be irreducible modules over Vir such that $d_k$ is locally finite on them for all $k>K$, $K$ is a fixed integer. Then $\bigotimes_{i=1}^{m}\Omega(\lambda_i, a_i, t_i )\bigotimes_{j=1}^{n}\Omega(\mu_j, \b_j )\bigotimes V\cong\bigotimes_{i'=1}^{m'}\Omega(\l'_{i'}, a'_{i'}, t'_{i'} )\bigotimes_{j'=1}^{n'}\Omega(\mu'_{j'}, \b'_{j'} )\bigotimes V'$ if and only if
$m=m', n=n', V\cong V', (\l_i, a_i)=(\l'_{i'}, a'_{i'}), (\mu_j, \b_j)=(\mu'_{j'}, \b'_{j'})$, where $i, i'=1,2,\cdots, m$, and $j, j'=1,2,\cdots, n$.
\begin{proof}
The sufficiency of the proposition is clear, we need only to prove the necessity. Denote $Y=\bigotimes_{i=1}^{m}\Omega(\lambda_i, a_i, t_i )\bigotimes_{j=1}^{n}\Omega(\mu_j, \b_j )\bigotimes V$ and $Y'=\bigotimes_{i'=1}^{m'}\Omega(\l'_{i'}, a'_{i'}, t'_{i'} )\bigotimes_{j'=1}^{n'}\Omega(\mu'_{j'}, \b'_{j'} )\bigotimes V'$,

 Let $\phi: Y\rightarrow Y'$ be a module isomorphism. From the part (2) of Lemma 4.1, we know that there is a vector space isomorphism $\tau : V\rightarrow V'$ such that $\phi(1\otimes v)=1\otimes \tau(v)$ for any $v\in V$. By Lemma 4.1, $\text{rank}(d_k(1\otimes v))=2m+n+1=\text{rank}(d_k(1\otimes\tau(v)))=2m'+n'+1$.\\

 \begin{claim}  For each $\l_i$ there is $i'$ such that $\l_i=\l'_{i'}$, for $i=1,2,\cdots,m$, and $i'=1,2,\cdots, m'$.
\end{claim}
Suppose $i=1$, $\l_1=\mu'_{j'}$ for some $j'\in Z_+$ and $1\leq j'\leq n'$. Since
\par
 $\phi(d_k(1\otimes v))=\sum_{i=1}^m\l_i^k\phi[(\partial_i+kG_i(1)-k^2a_i)\otimes v]+\sum_{j=1}^n\mu_j^k\phi[(\partial_{m+j}-k\b_j)\otimes v],$
\par
 $d_k\phi(1\otimes v)=\sum_{i'=1}^{m'}(\l'_{i'})^k[(\partial'_{i'}+kG'_{i'}(1)-k^2{a'}_{i'})
 \otimes\tau(v)]+\sum_{j'=1}^{n'}(\mu'_{j'})^{k}[({\partial}'_{m'+j'}-k\b'_{j'})\otimes\tau(v)].$
 \par
 $\phi(d_k(1\otimes v))-d_k\phi(1\otimes v)$=0. By Lemma 3.1,  the coefficient of $k^2\l_1^k$ is $\phi(-a_i\otimes v)=0$, then $a_i=0$, this is a contradiction.

Suppose $\l_1\neq \l'_{i'}$ for all $1\leq i'\leq m'$ to the contrary, $i'\in\Z_+$. Take $v\in V$, for $K\geq K(v)$, and $k, l>K$, we have
\begin{equation*}\begin{split}
  & (\l_1^{-k}d_k-\l_1^{-l}d_l)(1\otimes v)\\
= & \sum_{i=2}^m[(\frac{\l_i}{\l_1})^k-(\frac{\l_i}{\l_1})^l]\partial_i\otimes v+\sum_{i=1}^m[k(\frac{\l_i}{\l_1})^k-l(\frac{\l_i}{\l_1})^l]G_i(1)\otimes v +\sum_{i=1}^m[l^2(\frac{\l_i}{\l_1})^l-k^2(\frac{\l_i}{\l_1})^k]a_i\otimes v\\
 & +\sum_{j=1}^n[(\frac{\mu_j}{\l_1})^k-(\frac{\mu_j}{\l_1})^l]\partial_{m+j}\otimes v+\sum_{j=1}^n[l(\frac{\mu_j}{\l_1})^l-k(\frac{\mu_j}{\l_1})^k]\b_j\otimes v
\end{split}\end{equation*}

By Lemma 3.1 in the paper [TZ] and the properties of determinants (fixing $l$ and letting $k$ vary), we have that
 $$\text{rank}((\l_1^{-k}-\l_1^{-l}d_l)(1\otimes v_1): k,l>K)=2m+n.$$
But
\begin{equation*}\begin{split}
 & (\l_1^{-k}d_k-\l_1^{-l}d_l)(1\otimes \tau(v))\\
=& \sum_{i'=1}^{m'}[(\frac{\l'_{i'}}{\l_1})^k-(\frac{\l'_{i'}}{\l_1})^l]\partial'_{i'}\otimes\tau{(v)}
   +\sum_{i'=1}^{m'}[k(\frac{\l'_{i'}}{\l_1})^k-l(\frac{\l'_{i'}}{\l_1})^l]G'_{i'}(1)\otimes\tau{(v)}\\
&  +\sum_{i'=1}^{m'}[l^2(\frac{\l'_{i'}}{\l_1})^l-k^2(\frac{\l'_{i'}}{\l_1})^k]a'_{i'}\otimes\tau{(v)}
   +\sum_{j'=1}^{n'}[(\frac{\mu'_{j'}}{\l_1})^k-(\frac{\mu'_{j'}}{\l_1})^l]\partial'_{m'+j'}\otimes\tau{(v)}\\
&  +\sum_{j'=1}^{n'}[l(\frac{\mu'_{j'}}{\l_1})^l-k(\frac{\mu'_{j'}}{\l_1})^k]\b'_{j'}\otimes\tau{(v)}.
\end{split}
\end{equation*}
Using Lemma 3.1 in [TZ] again and the properties of determinants, we get
$\rank((\l_1^{-k}-\l_1^{-l}d_l)(1\otimes\tau(v): k,l>K)=2m'+n'+1$ which is a contradiction to $\text{rank}((\l_1^{-k}-\l_1^{-l}d_l)(1\otimes v): k,l>K)=2m+n$.
So $\l_1=\l'_{i'}$ for some $i'$, this completes Claim 1.

Since $\l_1, \l_2, \cdots, \l_m$ are pairwise distinct, we deduce that$\l_i=\l'_{i'}$ for each $i, i', 1\leq i, i'\leq m$, $i,i'\in\Z_+$, $m=m'$. for convenience, we denote $\l_i=\l'_i$ for each $i$. In the same way, we can have $\mu_j=\mu'_j$ for each $1\leq j\leq n, j\in\Z_+$.

\begin{claim}
$a_i=a'_i, \b_j=\b'_{j}$, for each $1\leq i\leq m, 1\leq j\leq n, i, j\in\Z_+$.
\end{claim}

Take a nonzero vector $v\in V$. For $k\geq \text{max}\{K(v), K(\tau(v))\}$, $\phi(d_k(1\otimes v))=d_k(\phi(1\otimes v))=d_k(1\otimes\tau(v))$,
we have
\begin{equation*}
\begin{split}
&  \sum_{i=1}^m\l_i^k[\phi(\partial_i\otimes v)-\partial'_i\otimes\tau(v)]+\sum_{i=1}^mk\l_i^k[\phi((a_i+t_i)\otimes v)-(a'_i+t'_i)\otimes\tau(v)]\\
+&  \sum_{i=1}^mk^2\l_i^k[a'_i\otimes\tau(v)-\phi(a_i\otimes v)]+\sum_{j=1}^n\mu_j^k[\phi(\partial_{m+j}\otimes v)-\partial'_{m+j}\otimes\tau(v)]\\
+& \sum_j^nk\mu_j^k[\b'_j\otimes\tau(v)-\phi(\b_j\otimes v)]=0.
\end{split}
\end{equation*}
By Lemma 3.1 of [TZ], $\phi(\partial_i\otimes v)=\partial'_i\otimes\tau(v)$, $a_i=a'_i$ and $\phi((a_i+t_i)\otimes v)=(a'_i+t'_i)\otimes\tau(v)$, i.e.,$\phi(t_i\otimes v)=t'_i\otimes\tau(v)$, $\phi(\partial_{m+j}\otimes v)=\partial'_{m+j}\otimes\tau(v), \b'_j=\b_j$, $1\leq i\leq m$. This completes Claim 2.

\begin{claim}$\tau$ is a module isomorphism.
\end{claim}

We need only to prove $\tau$ is a module homomorphism.
By Claim 1 and Claim 2 and $\phi(d_k(1\otimes v))=d_k(1\otimes\tau(v))$, we deduce that $\phi(1\otimes d_k(v))=1\otimes d_k(\tau(v))$,
yielding that $\tau d_k(v)=d_k\tau(v)$. At the same time, $\phi(c(1\otimes v))=1\otimes\tau(c v)$, $c\phi((1\otimes v))=c(1\otimes\tau(v))=1\otimes c\tau(v)$, so $\tau(c v)=c\tau(v)$, therefore, $\tau: V\rightarrow V'$ is a module homomorphism, this complete Claim 3.

\end{proof}

\end{proposition}

\section{New irreducible Virasoro module}

In this section we will compare the irreducible tensor products with
all other known non-weight irreducible Virasoro modules in
\cite{LZ1, LLZ, MZ2, MW}, and \cite{TZ1}.
For any $s\in \mathbb{Z}_{+}, l,m\in \mathbb{Z},$ as in \cite{LLZ},
we denote
$$ \omega^{(s)}_{l,m} = \sum^{s}_{i=0}\binom{s}{i}(-1)^{s-i}d_{l-m-i}d_{m+i} \in U(\Vir).$$
\begin{lemma} \label{2}
 Let $\lambda_{j_1},\mu_{j_2}\in \mathbb{C}^{*}$, $\deg(h_{j_1}) = 1$ and $\alpha_{j_2},\b_{j_2}\neq 0.$ Let $V$ be an infinite dimensional
irreducible $\Vir$ module such that each $d_{k}$ is locally finite on $V$ for all
$k \geq R$ for a fixed $R \in \mathbb{N}.$
\begin{enumerate}
\item \label{5.1.1}
For any positive integer $n,$ the action of $\Vir^{(n)}_{+}$ on
$\bigotimes_{j_1=1}^{n_1}\Omega(\lambda_{j_1}, \alpha_{j_1}, h_{j_1})\bigotimes_{j_2=1}^{n_2}(\mu_{j_2},\b_{j_2})$
$\otimes V$ is not locally finite.
\item \label{5.1.2}
Suppose $h_{j_1}=\xi_{j_1} t_{j_1}+\eta_{j_1}, \xi_{j_1}\neq0$.
$$W=\text{span}\{\partial_{1}^{r_1}t_{1}^{p_1}\otimes\cdots\otimes\partial_{n_1}^{r_{n_1}}t_{n_1}^{p_{n_1}}\otimes\partial_{{n_1}+1}^{r_{{n_1}+1}}
\otimes\cdots\otimes\partial_{{n_1}+n_2}^{r_{{n_1}+n_2}}:\text{for some}\quad r_i>s, 1\leq i\leq n_1+n_2\},$$
we have
$\omega^{(s)}_{l,m}(W)$ is not zero,$\ \forall\ l,m\in \mathbb{Z}, s,r_{i}\in \mathbb{Z_{+}}.$
\item \label{5.1.3}
For any integer $s > 4,$ there exists $v\in V,m,l\in
\mathbb{Z}$ such that in $\bigotimes_{{j_1}=1}^{n_1}\Omega(\lambda_{j_1}, \alpha_{j_1}, h_{j_1})$\\
$\bigotimes_{j_2=1}^{n_2}$$(\mu_{j_2},\b_{j_2})\otimes V$ we have
$  \omega^{(s)}_{l,-m}(1\otimes v) \neq 0.$
\end{enumerate}
\end{lemma}
\begin{proof} (1).  For any $v\in V$ and
$k\in\Z_+$, it is clear that
$d_{n+1}^1(1\otimes\cdots\otimes1\otimes v),\cdots d_{n+1}^k(1\otimes\cdots\otimes1\otimes v)$ are linearly
independent. So the assertion follows easily.

\noindent(2). For any $f\in W$, we have
\begin{equation*}
\begin{split}
  & \omega^{(s)}_{l,m}(f)
=  \sum^{s}_{i=0}\binom{s}{i}(-1)^{s-i}d_{l-m-i}d_{m+i}(f)\\
= & \sum^{s}_{i=0}\binom{s}{i}(-1)^{s-i}d_{l-m-i}\Big(\sum_{(r,p)\in S}\sum_{j_1=1}^{n_1}\partial_{1}^{r_1}t_{1}^{p_1}\otimes\cdots\otimes
    d_{m+i}(\partial_{j_1}^{r_{j_1}}t_{j_1}^{p_{j_1}})\otimes\cdots\otimes\partial_{n_1}^{r_{n_1}}t_{n_1}^{p_{n_1}}\otimes\partial_{{n_1}+1}^{r_{{n_1}+1}}\otimes\cdots\\
  & \hskip0.5cm  \otimes\partial_{{n_1}+n_2}^{r_{{n_1}+n_2}}
    +\sum_{(r,p)\in S}\sum_{j_2=1}^{n_2}\partial_{1}^{r_1}t_{1}^{p_1}\otimes\cdots\otimes \partial_{n_1}^{r_{n_1}}t_{n_1}^{p_{n_1}}\otimes\partial_{{n_1}+1}^{r_{{n_1}+1}}\otimes\cdots\otimes d_{m+i}(\partial_{{n_1}+j_{2}}^{r_{{n_1}+j_{2}}})
    \otimes\cdots\otimes\partial_{{n_1}+n_2}^{r_{{n_1}+n_2}}\Big)\\
= & \sum^{s}_{i=0}\binom{s}{i}(-1)^{s-i}\Big(\sum_{(r,p)\in S}\big(\sum_{j_{1}=1}^{n_1}\sum_{k_{1}\neq j_{1}}^{n_1}\partial_{1}^{r_1}t_{1}^{p_1}\otimes\cdots\otimes
    d_{l-m-i}(\partial_{k_1}^{r_{k_1}}t_{k_1}^{p_{k_1}})\otimes\cdots\otimes d_{m+i}(\partial_{j_1}^{r_{j_1}}t_{j_1}^{p_{j_1}})\\
  & \hskip8cm   \otimes\cdots\otimes\partial_{n_1}^{r_{n_1}}t_{{n_1}}^{p_{n_1}}\otimes\partial_{{n_1}+1}^{r_{{n_1}+1}}\otimes\cdots\otimes\partial_{n_1+n_2}^{r_{n_1+n_2}}\\
  & +\sum_{j_{1}=1}^{n_1}\partial_{1}^{r_1}t_{1}^{p_1}\otimes\cdots\otimes
    d_{l-m-i}d_{m+i}(\partial_{j_1}^{r_{j_1}}t_{j_1}^{p_{j_1}})\otimes\cdots\otimes\partial_{m}^{r_m}t_{m}^{p_m}\otimes\partial_{m+1}^{r_{m+1}}
    \otimes\cdots\otimes\partial_{{n_1}+n_2}^{r_{{n_1}+n_2}}\\
  & +\sum_{j_{1}=1}^{n_1}\sum_{j_{2}=1}^{n_2}\partial_{1}^{r_1}t_{1}^{p_1}\otimes\cdots\otimes
    d_{m+i}(\partial_{j_1}^{r_{j_1}}t_{j_1}^{p_{j_1}})\otimes\cdots\otimes\partial_{n_1}^{r_{n_1}}t_{{n_1}}^{p_{n_1}}\otimes\partial_{{n_1}+1}^{r_{{n_1}+1}}
    \otimes\cdots\otimes d_{l-m-i}(\partial_{{n_1}+j_{2}}^{r_{{n_1}+j_{2}}})\\
  & \hskip8cm  \otimes\cdots\otimes\partial_{{n_1}+n_2}^{r_{{n_1}+n_2}}\big)\\
  & +\sum_{(r,p)\in S}\big(\sum_{j_{2}=1}^{n_2}\sum_{k_{2}\neq j_{2}}^{n_2}\partial_{1}^{r_1}t_{1}^{p_1}\otimes\cdots\otimes
    \partial_{n_1}^{r_{n_1}}t_{n_1}^{p_{n_1}}\otimes\partial_{{n_1}+1}^{r_{{n_1}+1}}\otimes\cdots\otimes d_{l-m-i}(\partial_{{n_1}+k_2}^{r_{{n_1}+k_2}})
    \otimes\cdots\otimes d_{m+i}(\partial_{{n_1}+j_2}^{r_{{n_1+j_2}}})\\
  & \hskip8cm   \otimes\cdots\otimes\partial_{{n_1}+n_2}^{r_{{n_1}+n_2}}\\
  & +\sum_{j_{2}=1}^{n_2}\partial_{1}^{r_1}t_{1}^{p_1}\otimes\cdots\otimes \partial_{n_1}^{r_{n_1}}t_{n_1}^{p_{n_1}}\otimes\partial_{{n_1}+1}^{r_{{n_1}+1}}
     \otimes\cdots\otimes d_{l-m-i}d_{m+i}(\partial_{{n_1}+j_2}^{r_{{n_1+j_2}}})\otimes\cdots\otimes\partial_{{n_1}+n_2}^{r_{{n_1+n_2}}}\\
  & +\sum_{j_{2}=1}^{n_2}\sum_{j_{1}=1}^{n_1}\partial_{1}^{r_1}t_{1}^{p_1}\otimes\cdots\otimes
    d_{l-m-i}(\partial_{j_1}^{r_{j_1}}t_{j_1}^{p_{j_1}})\otimes\cdots\otimes\partial_{n_1}^{r_{n_1}}t_{{n_1}}^{p_{n_1}}\otimes\partial_{{n_1}+1}^{r_{{n_1}+1}}
    \otimes\cdots\otimes d_{m+i}(\partial_{{n_1}+j_{2}}^{r_{{n_1}+j{2}}})\\
  & \hskip8cm   \otimes\cdots\otimes\partial_{{n_1}+n_2}^{r_{{n_1}+n_2}}\big)\Big).\\
\end{split}
\end{equation*}
Since
\begin{equation*}
\begin{split}
 &   d_{l-m-i}(\partial_{j_1}^{r_{j_1}}t_{j_1}^{p_{j_1}})
   = \l_{j_1}^{l-m-i}(\partial_{j_1}-l+m+i)^{r_{j_1}}\Big(\partial_{j_1}t_{j_1}^{p_{j_1}}+(l-m-i)G_{j_1}(t_{j_1}^{p_{j_1}})-(l-m-i)^2\\
 &   \hskip9cm  \a_{j_1}F_{j_1}(t_{j_1}^{p_{j_1}})\Big);\\
 &   d_{m+i}(\partial_{n_1+j_2}^{r_{n_1+j_2}})
   = \mu_{j_2}^{m+i}(\partial_{n_1+j_2}-m-i)^{r_{n_1+j_2}}\Big(\partial_{n_1+j_2}-(m+i)\b_{j_2}\Big),\\
 &   d_{m+i}(\partial_{j_1}^{r_{j_1}}t_{k_1}^{p_{j_1}})
   = \l_{j_1}^{m+i}(\partial_{j_1}-m-i)^{r_{j_1}}\Big(\partial_{j_1}t_{k_1}^{p_{j_1}}+(m+i)G_{j_1}(j_{k_1}^{p_{j_1}})-
     (m+i)^2\a_{j_1}F_{j_1}(t_{j_1}^{p_{j_1}})\Big);\\
\end{split}
\end{equation*}
\begin{equation*}
\begin{split}
 &   d_{l-m-i}d_{m+i}(\partial_{j_1}^{r_{j_1}}t_{j_1}^{p_{j_1}})
   = \l_{j_1}^{l}(\partial_{j_1}-l)^{r_{j_1}}\Big((\partial_{j_1}-l+m+i)\big(\partial_{j_1}t_{j_1}^{p_{j_1}}+(l-m-i)G_{j_1}(t_{j_1}^{p_{j_1}})\\
 &  \hskip10cm  -(l-m-i)^2\a_{j_1}F_{j_1}(t_{j_1}^{p_{j_1}})\big)\\
 &  \hskip2cm  +(m+i)\big(\partial_{j_1}G_{j_{1}}(t_{j_1}^{p_{j_1}})+(l-m-i)G^{2}_{j_1}(t_{j_1}^{p_{j_1}})-(l-m-i)^2
               \a_{j_1}F_{j_1}G_{j_1}(t_{j_1}^{p_{j_1}})\big)\\
 &  \hskip2cm  -(m+i)^{2}\a_{j_{1}}\big(\partial_{j_1}F_{j_1}(t_{j_1}^{p_{j_1}})+(l-m-i)G_{j_1}F_{j_1}(t_{j_1}^{p_{j_1}})-(l-m-i)^2
               \a_{j_1}F^{2}_{j_1}(t_{j_1}^{p_{j_1}})\big)\Big)\\
 &   d_{l-m-i}(\partial_{n_1+j_2}^{r_{n_1+j_2}})
   =  \mu_{j_2}^{l-m-i}(\partial_{n_1+j_2}-l+m+i)^{r_{n_1+j_2}}\Big(\partial_{n_1+j_2}-(l-m-i)\b_{j_2}\Big);\\
 &   d_{l-m-i}d_{m+i}(\partial_{n_1+j_2}^{r_{n_1+j_2}})=
     \mu_{j_2}^{l}(\partial_{n_1+j_2}-l)^{r_{n_1+j_2}}\Big((\partial_{n_1+j_2}-l+m+i)\big(\partial_{n_1+j_2}\\
 &  \hskip4cm  -(l-m-i)\b_{j_2}\big)-(m+i)\b_{j_2}\big(\partial_{n_1+j_2}-(l-m-i)\b_{j_2}\big)\Big)\\
\end{split}
\end{equation*}
By computing above, there must be some items in which the maximum power of $i$ is higher than $s$ in $\omega_{l,m}^{(s)}(f)$,
then using the following identity
\begin{equation}\label{identity} \sum_{i=0}^r(-1)^{r-i}{r\choose
i}i^j=0,\ \forall\ j, r\in\Z_+ {\rm\ with\ } j<r,
\end{equation}
we easily get
$\omega^{(s)}_{l,m}(f)$ is not zero.

\noindent(3). Fix any $s>4$. Take $v$ to be a highest weight vector
in $V$ if $V$ is a highest weight module, otherwise $v$ can be any
nonzero vector in $V.$ From \cite{FF} and \cite{MZ2} we know that
the vectors $v, d_{-2}v, d_{-3}v,\cdots, d_{-s-2}v$ are linearly
independent in $V$ and there exists $K\in\N$ such that these vectors
are annihilated by $d_{m}$ for all $m>K$. For any $l>K$ and $m=s+2$,
we can easily compute that $\omega^{(s)}_{l,-m}(1) = 0$
provided $s>4,$ by lemma \ref{2} \eqref{5.1.2}.
Hence we also deduce that
\begin{equation*}
\begin{split}
\omega^{(s)}_{l,-m}(1\otimes v) = & \sum^{s}_{i=0}\binom{s}{i}(-1)^{s-i}d_{l+m-i}d_{-m+i}(1\otimes v)\\
 = & \sum^{s}_{i=0}\binom{s}{i}(-1)^{s-i}d_{l+m-i}(1)\otimes d_{-m+i}(v)\\
\end{split}
\end{equation*}
which is nonzero.
\end{proof}

\begin{theorem}
The $\Vir$-modules $\bigotimes_{i=1}^{m}\Omega(\l_{i},\a_{i},h_{i})\bigotimes_{j=1}^{n}\Omega(\mu_j,\b_j)$ or
$\bigotimes^{m}_{i=1}\Omega(\lambda_{i},\alpha_{i},h_{i})$\\
$\bigotimes_{j=1}^{n}\Omega(\mu_j,\b_j)\otimes V$ is not isomorphic to any
irreducible module defined in \cite{MZ2, LZ1, LLZ}.
\end{theorem}

\begin{proof} For any irreducible modules in \cite{MZ2}, there
exists a positive integer $n$ such that $d_n$ acts locally finitely.
So neither $\bigotimes^{m}_{i=1}\Omega(\l_{i},\a_{i}, h_{i})\bigotimes_{j=1}^{n}\Omega(\mu_j,\b_j)$ nor
$\bigotimes^{m}_{i=1}\Omega(\l_{i},\a_{i},h_{i})$
$\bigotimes_{j=1}^{n}\Omega(\mu_j,\b_j)\otimes
V$ is isomorphic to any irreducible modules constructed in
\cite{MZ2} by Lemma \ref{2} \eqref{5.1.1}.

Then we consider the irreducible non-weight $\Vir$-module $A_b$
defined in \cite{LZ1}. From the proof of Theorem 9 in \cite{LLZ} or
the argument in the proof of Corollary 4 in \cite{TZ1}, we have
\begin{equation}\label{omega_A}
\omega^{(r)}_{l,m}(A_{b}) = 0,\ \forall\ l, m \in \mathbb{Z}, r\geq
3. \end{equation}
Combining this with Lemma \ref{2} \eqref{5.1.2} and \eqref{5.1.3},
we see easily that $\bigotimes_{i=1}^{m}\Omega(\l_{i},\a_{i},h_{i})\bigotimes_{j=1}^{n}\Omega(\mu_j,$
$\b_j)\not\cong A_b$ and
$\bigotimes_{i=1}^{m}\Omega(\l_{i},\a_{i},h_{i})\bigotimes_{j=1}^{n}\Omega(\mu_j,\b_j)\otimes V\not\cong A_b$

Now we recall the irreducible non-weight Virasoro modules defined in
\cite{LLZ}. Let $M$ be an irreducible module over the Lie algebra
$\mathfrak{a}_{k} := \Vir_{+}/ \Vir^{(k)}_{+}, k \in \mathbb{N}$
such that the action of $\bar{d_{k}} := d_{k}+\Vir^{(k)}_{+}$ on $M$
is injective, where $\Vir^{(k)}_{+}=\{d_i\ |\ i>k\}$ and
$\Vir_{+}=\spn\{d_i\ |\ i\in\Z_+\}$. For any $\beta\in
\mathbb{C}[t^{\pm1}]\setminus \mathbb{C}$, the $\Vir$-module
structure on $\mathcal{N}(M,\beta) = M\otimes \mathbb{C}[t^{\pm1}]$
is defined by
\begin{equation*}\begin{array}{l}
d_{m}\cdot (v\otimes t^{n}) = (n +
\sum^{k}_{i=0}\frac{m^{i+1}}{(i+1)!}\bar{d_{i}})v\otimes t^{n+m} +v\otimes (\beta t^{m+n}),\\\\
 c\cdot (v\otimes t^{n}) = 0, \ \forall\ m, n \in \mathbb{Z}.
\end{array}\end{equation*} From the computation in $(6.7)$
of \cite{LLZ} we see that
\begin{equation}\label{omega_N(M)}\begin{array}{l}
\omega^{(r)}_{l,m}(\mathcal{N}(M,\beta)) = 0,\ \forall\ l,m \in \mathbb{Z}, r > 2k+2,\\\\
\omega^{(2k+2)}_{l,m}(w\otimes t^{i}) = (2k+2)!(-1)^{k+1}(\bar{d_k}
^2 w)\otimes t^{i+l}\neq 0,\ \forall\ l,m\in\Z.
\end{array}\end{equation}
Combining the first equation of \eqref{omega_N(M)} with Lemma \ref{2} \eqref{5.1.2} and \eqref{5.1.3}, we see that $\mathcal{N}(M,\beta)$ is not isomorphic
to $\bigotimes_{i=1}^{m}\Omega(\l_{i},\a_{i},h_{i})\bigotimes_{j=1}^{n}\Omega(\mu_j,\b_j)$ and
$\bigotimes_{i=1}^{m}\Omega(\l_{i},\a_{i},h_{i})\bigotimes_{j=1}^{n}\Omega(\mu_j,\b_j)\otimes V$. If
$k=1$, we see that $u$ and $\omega^{(4)}_{l,m}(u)$ are linearly
independent for any $u\in\mathcal{N}(M,\beta)$ provided  $l\neq 0$,
while
$\omega^{(4)}_{l,m}(1\otimes\cdots\otimes1)=\big(\sum^{m}_{i=1}24\l_{i}^{l}\a_{i}^{2}\xi_{i}^{2}+
\sum^{m}_{i=1}\sum^{m}_{j\neq i}24\l_{i}^{l-m-i}\l_{j}^{m+i}\a_{i}^{2}\xi_{i}^{2}\a_{j}^{2}\xi_{j}^{2}\big)(1\otimes\cdots\otimes1)$
in $\bigotimes_{i=1}^{m}\Omega(\l_{i},\a_{i},h_{i})\bigotimes_{j=1}^{n}\Omega(\mu_j,\b_j)$ by Lemma \ref{2} \eqref{5.1.2}, so that
$\bigotimes_{i=1}^{m}\Omega(\l_{i},\a_{i},h_{i})\bigotimes_{j=1}^{n}\Omega(\mu_j,\b_j)\not\cong\mathcal{N}(M,\beta)$
in this case. This complete the Theorem.
\end{proof}
\par

Finally, we compare our tensor product modules are not isomorphic 
to the tensor products modules defined in [TZ2]. 
%

\begin{theorem} Suppose that $m\neq0$.
The $\Vir$-modules $\bigotimes_{i=1}^{m}\Omega(\l_{i},\a_{i},h_{i})\bigotimes_{j=1}^{n}\Omega(\mu_j,\b_j)$ or
$\bigotimes^{m}_{i=1}\Omega(\lambda_{i},\alpha_{i},h_{i})$
$\bigotimes_{j=1}^{n}\Omega(\mu_j,\b_j)\otimes V$ is not isomorphic to any
module $\bigotimes_{k=1}^l\Omega(\mu_k, \b_k)\bigotimes V'$. 
\end{theorem}

\begin{proof} The result is obvious since the modules in question are not isomorphic as $\C[d_0]$-modules.
\end{proof}

\end{document}